\documentclass[preprint,sort&compress,times,3p]{elsarticle}%preprint
\usepackage[latin1]{inputenc}

\usepackage{amsmath,amsthm,amssymb,amsfonts,mathtools}
\usepackage{array}

\usepackage{booktabs}
\usepackage{xcolor}

\usepackage[caption=false]{subfig}

\usepackage{fancyhdr}
\newlength\tabwidth
\makeatletter
\newcommand{\dashrule}[1][black]{%
  \color{#1}\rule[\dimexpr.5ex-.2pt]{4pt}{.4pt}\xleaders\hbox{\rule{4pt}{0pt}\rule[\dimexpr.5ex-.2pt]{4pt}{.4pt}}\hfill\kern0pt%
}
\newcommand{\rulecolor}[1]{%
  \def\CT@arc@{\color{#1}}%
}
\makeatother
\usepackage{hyperref}

\newcolumntype{L}{>{$\displaystyle}l<{$}}

\newcommand{\beq}{\begin{equation}}
\newcommand{\eeq}{\end{equation}}

 \newcommand{\R}{\mathbb{R}}

 \newcommand{\eps}{\varepsilon}

 \newcommand{\xvec}{\mathbf{r}}

\newcommand{\cO}{\ensuremath{\mathcal{O}}}

\DeclareMathOperator\bch{bch}

\newcommand{\Frac}[2]%
     {\frac{\displaystyle #1}{\displaystyle #2}}

\usepackage{natbib}
\newtheorem{theorem}{Theorem}[section]
\newtheorem{lemma}[theorem]{Lemma}

\def\ii{i}
\def\beqs{\begin{equation*}}
\def\eeqs{\end{equation*}}

\begin{document}

\title{Fourier-Splitting methods for the dynamics of rotating Bose-Einstein condensates}
\author{Philipp Bader}
\ead{p.bader@latrobe.edu.au}
\address{Department of Mathematics and Statistics, La Trobe University, Bundoora 3068 VIC, Australia}
%Instituto de Matem\'atica Multidisciplinar,  Universitat Polit\`ecnica de Val\`{e}ncia, E-46022  Valencia, Spain.}

\begin{abstract}
	We present a new method to propagate rotating Bose-Einstein condensates subject to explicitly time-dependent trapping potentials. Using algebraic techniques, we combine Magnus expansions and splitting methods to yield any order methods for the multivariate and nonautonomous quadratic part of the Hamiltonian that can be computed using only Fourier transforms at the cost of solving a small system of polynomial equations. 
The resulting scheme solves the challenging component of the (nonlinear) Hamiltonian and can be combined with optimized splitting methods to yield efficient algorithms for rotating Bose-Einstein condensates.
The method is particularly efficient for potentials that can be regarded as perturbed rotating and trapped condensates, e.g., for small nonlinearities, since it retains the near-integrable structure of the problem.
For large nonlinearities, the method remains highly efficient if higher order $p>2$ is sought. Furthermore, we show how it can be adapted to the presence of dissipation terms.
Numerical examples illustrate the performance of the scheme.
\end{abstract}

\maketitle
%\begin{AMS}
%35Q55,  % NONLINEAR SE
%65T99,  % NA: NUMERICAL METHODS WITH FOURIER ANALYSIS
%65Z05,  % NA: Application to physics
%65N35,  % NA: Spectral, collocation and related methods
%81-08   % Quantum theory: computational methods
%\end{AMS}

\section{Introduction}
The centerpiece of this work is the construction of an efficient geometric integrator for the two-dimensional harmonically trapped rotational Schr\"odinger equation in atomic units ($\hbar=m=1$) subject to periodic boundary conditions
\newcommand{\pvec}{\mathbf{p}}
\newcommand{\tn}{t}
\begin{align}\label{eq:1}
		i\partial_t \psi(\xvec,t) &= H_A(t)\psi(\xvec,t) , \qquad \psi(\xvec,0)=\psi_0\in L^2([-\pi,\pi]^2),
		%+ \varepsilon V(t, x,y,z) + g(t)|\psi(x,t)|^2
\\
		\intertext{with the explicitly time-dependent Hamiltonian } \nonumber
		H_A(t)&=
\tfrac12 \pvec^T\pvec + \tfrac12\left(\omega_x(t)^2 x^2 + \omega_y(t)^2 y^2\right)
    + \Omega L_z,
\end{align}
%with square integrable initial condition $\psi(\xvec,0)=\psi_0\in L^2([-1,1]^3)$, 
where $\xvec=(x,y)^T$, $\pvec=(p_x,p_y)^T$, $L_z = xp_y - yp_x$ denotes the angular momentum operator and $p_{k}=-i\partial_{k},\;k=x,y$.
This includes the case of unbounded domains since the solution vanishes up to round-off at sufficiently large spatial intervals due to the harmonic trapping potential. 
For simplicity of the presentation, we have chosen a simple form of the Hamiltonian \eqref{eq:1}, but our methodology also applies to virtually all relevant polynomial Hamiltonians of degree $\leq$ 2 in any dimension with arbitrary time-dependencies and we will show how to extend the presented techniques for more general quadratic and linear time-dependencies which are used to model collisions of atoms and molecules \cite{recamier85eti,fernandez89mtp}. 
%\red{Liu liu09efa is a strange author - physica scripta and hotmail address, china. at least it is recent (2009) (but has no citation)}
%
%In fact, the interval limits have to be adjusted for each potential in order to get reasonable numerical results.
The generalization to three dimensions is straightforward and will be briefly addressed in section~2.

The efficient solution of \eqref{eq:1} is of paramount importance to the computation of the dynamics of rotating Bose-Einstein condensates as we will see below, and in contrast to previous efforts \cite{chin05foa,bao06dor,wang07ats,bao09agl,bao13asa,ming14aes,hofstaetter14cos}, time-dependent (trapping) potentials and non-linearities can be treated without tempering the algebraic structure of the problem. The presence of such time-dependencies impedes a simple transformation to a rotating system of coordinates which would eliminate the rotation term $L_z$ for autonomous $H_A$.

At any given time $t$ and for \textit{any order} $p>1$, we show that, for a sufficiently small time-step $h$, there exist cheaply computable coefficients $f_j(t,h),g_k(t,h), e_l(t,h) \in i\mathbb{R}$ obeying a small system of polynomial equations such that
\beq\label{eq:decomp}
	e^{f_0 x^2}
	e^{f_1 y^2 + g_1 p_x^2 - e_1 yp_x}
	e^{f_2 x^2 + g_2 p_y^2 + e_2 xp_y}
	e^{f_3 y^2 + g_3 p_x^2 - e_3 yp_x}
	= \varphi^{H_A}_{\tn,\tn+h} +\cO\left({h^{p+1}}\right),
\eeq
where $\varphi^{H_A}_{t,t+h}$ denotes the exact flow of \eqref{eq:1} from $t$ to $t+h$.
By virtue of this decomposition, named $\Phi_{t,t+h}^{[p]}$, the position and moment coordinates are decoupled and can be diagonalized using Fourier transforms. After discretization, only six (one-dimensional) changes from coordinate to momentum space and vice versa per time-step exponents are required. These changes are performed by Fast Fourier Transforms (FFT) and hence suggest the name \emph{Fourier-splitting}.
The approximation preserves \emph{unitarity} (and thus the $L_2$-norm) and \emph{gauge invariance} of the exact solution and hence, it can be considered a geometric integrator in the sense of Ref.~\cite{faou12gni}. Furthermore, one can associate a time-dependent Hamiltonian with the decomposition which is exactly solved at each step.

The method is particularly successful for perturbed problems of the form 
\[
	H= H_A(t) + \varepsilon B(t, \xvec, |\psi|), \qquad \varepsilon\ll1,
\]
with a small parameter $\varepsilon$, and some real-valued function $B$,
which includes the Gross-Pitaevskii equation for Bose-Einstein condensates as special case.
The (nonlinear) Hamiltonian $H$ with $B=g|\psi|^2+V$ describes the evolution of a rotating Bose-Einstein condensate (BEC) subject to a harmonic (parabolic) trapping potential plus some perturbation $\varepsilon V$.
After the first experimental realization of BECs %\cite{setBEC}
\cite{bec-experiment_1,bec-experiment_2,bec-experiment_3} and the consequently awarded Nobel prize in 2001, continuous attention of numerical analysts \cite{chin05foa,bao06dor,wang07ats,bao09agl,bao13asa,hofstaetter14cos,ming14aes} %(cf. references) 
has been drawn to the solution of the autonomous version of \eqref{eq:1}, which is obtained by dropping all time-dependencies in the Hamiltonian. 

The flow of the perturbation $B$ can be easily computed since $B$ is diagonal in coordinate space and leaves the modulus $|\psi|$ constant, see Lemma~\ref{lemma} for details.
Using \eqref{eq:decomp}, the exact flow can be approximated by Strang's method to
\beq\label{strang}
	\varphi^{\varepsilon\tilde{B}(\tn+h)}_{h/2} \circ \Phi_{\tn,\tn+h}^{[p]} \circ \varphi^{\varepsilon\tilde{B}(\tn)}_{h/2} = \varphi^{H}_{\tn,\tn+h} + \cO\left({\varepsilon h^3} + h^{p+1}\right),
\eeq
where the tildes, $\tilde{B}$, indicate frozen (nonlinear) operators, i.e., $\varphi_{h}^{\tilde{B}(s)}$ is the flow of $i\dot{u}(t)=B(s)u(t)$.
%\red{
The term proportional to $h^{p+1}$ originates from the error in the approximation of the part $H_A$ by the $p$th order method $\Phi^{[p]}$ \eqref{eq:decomp}.
%}
Observe that the outer exponentials of \eqref{eq:decomp} are diagonal in coordinate space and no further FFT is necessary to solve the full problem \eqref{strang}.
An alternative approach \cite{bao06dor,wang07ats} splits the system into simultaneously diagonalizable parts 
$T_x=\tfrac12 p_x^2 - \Omega yp_x$, $T_y=\tfrac12 p_y^2 + \Omega xp_y$, $W=\tfrac{1}{2}\left(\omega_x(t)^2x^2+\omega_y(t)^2y^2\right)+\varepsilon B(t)$ and then
\beq\label{strang:classic}
	\varphi^{\tilde{W}(\tn+h)}_{h/2} \circ \varphi^{T_x}_{h/2}\circ\varphi^{T_y}_{h}\circ\varphi^{T_x}_{h/2} \circ \varphi^{\tilde{W}(\tn)}_{h/2} = \varphi^{H}_{\tn,\tn+h}+ \cO\left({h^3}\right),
\eeq
which also requires six FFTs but the small factor $\varepsilon$ in the error is lost. 
If the time is frozen in $H_A$, Laguerre transforms \cite{bao09agl,bao13asa,ming14aes,hofstaetter14cos} or a decomposition similar to \eqref{eq:decomp} \cite{chin05foa} can be used to advance $H_A$ \emph{without} recovering the small factor and even lose the  property $[B,[B,[B,H_A]]]=0$ which simplifies the design of highly efficient splitting methods \cite{blanes10sac}.

Eventually, the method will be embedded in such a splitting framework that generalizes \eqref{strang} and by comparing with \eqref{strang:classic}, it becomes clear that the number of flows $\varphi$ that have to be treated individually is reduced to two which will enable us to use optimized splitting methods from the literature. In consequence, we will see in the numerical experiments that the new procedure is efficient even for 
\[
	H= H_A(t) + \varepsilon B(t, \xvec, |\psi|), \qquad \varepsilon\gg1.
	\]
The decomposition is built upon earlier works for rotating but autonomous BEC \cite{chin05foa} and explicitly time\--dependent one-dimensional harmonic oscillators \cite{bader11fmf}, where Fourier-splittings have been used for simpler Hamiltonians.

In the following section, we give a short introduction to some numerical concepts which will culminate in the derivation of our method.
As described, the method addresses the solution of the dominant part in the Hamiltonian, i.e., kinetic energy, trapping and rotation, $H_A$. 
Its form is closely related to a splitting method, in fact, if the coefficients $f,g,e$ were taken to be 
\[
	f_0=0, \ f_1 =\frac14 \omega_y(t)^2, \ g_1=\frac14, \ e_1=\frac12\Omega,\
	f_2=\frac12 \omega_x(t)^2, \ g_2=\frac12, \ e_2=\Omega, \ f_3=f_1, \ g_3=g_1,\ e_3=e_1,
\]
we would recover a second order Strang splitting. We show how to modify these scalar coefficients in order yield an any-order approximation using the same number of exponentials.
Once we have established how to solve this part of the Hamiltonian as a whole, we can use it as building block in a splitting method for nonlinear Hamiltonians or in the presence of (time-dependent) perturbations.

In such perturbative settings, the algorithm can demonstrate its efficiency as seen from \eqref{strang} and \eqref{strang:classic} because an additional factor $\eps$ can be gained in the error.  We will elaborate on splittings for such near-integrable systems subject to explicit time-dependencies since the time coordinate has to be treated in a particular way in order to preserve the smallness in the error.

It turns out, that the rotation Hamiltonian $H_A(t)$ cannot be frozen, as $e^{-ihH_A(t)}$, but has to be propagated accurately in time using the flow $\varphi_{t,t+h}^{H_A}$.
For this purpose, we introduce the Magnus expansion that will produce an approximation to the exact flow to any desired order using only manipulations in the algebra generated by the Hamiltonian. 

Since this algebra is finite dimensional, all commutators in the Magnus expansion can be expressed in a simple basis which will then be used to construct our decomposition.

The efficiency of the method is demonstrated by a series of numerical examples.

\section{Derivation of the new method}
For the construction of our algorithm \eqref{eq:decomp}, a variety of tools are employed which will be briefly discussed in this section.

Splitting methods are frequently recommended for the integration of (nonlinear) Schr\"odinger equations due to
 their fast computability and high accuracy \cite{bao05afo,perez03amc,thalhammer09jcp}. 
Furthermore, they preserve geometric features of the exact solution, such as norm-conservation (unitarity) and gauge-invariance. 
\subsection{Splitting methods}
For a Hamiltonian $H = A + B$, suppose that the flows for one time-step $h$ of the parts $A, B$ are available as $\varphi^A_h$ and $\varphi^B_h$, respectively, then, a $p$th-order $s$-stage approximation $\Psi_h^{[p,s]}$ of the full solution $\varphi^H$ can be computed as
\beq\label{eq.splitting}
	\Psi_h^{[p,s]} = \varphi^{B}_{b_1 h}\circ \varphi^{A}_{a_1 h}
						\circ\varphi^{B}_{b_2 h}\circ \varphi^{A}_{a_2 h}\cdots
						\circ\varphi^{B}_{b_s h}\circ \varphi^{A}_{a_s h}
						= \varphi^{H}_h+ \mathcal{O}\left(h^{p+1}\right),
\eeq
for a suitable choice of coefficients $a_j,b_j\in\mathbb{R}$.
The choice $a_1=1, a_2=0$ and $b_1=b_2=1/2$ with $s=2$ corresponds to Strang's second-order method \eqref{strang}.
Higher-order methods can be designed using the Baker-Campbell-Hausdorff (BCH) formula, $e^{hA}e^{hB}=e^{\bch(hA,hB)}$, whose first terms are given by
\beq\label{eq.bch}
	\bch(hA,hB) = h(A+B) + \tfrac{h^2}{2}[A,B] + \tfrac{h^3}{12}\left([A,[A,B]]-[B,[A,B]]\right) + \mathcal{O}(h^4),
\eeq
for operators $A, B$ in some Lie algebra. 
The formula allows to derive a modified vector field, $h\tilde{H}=h(A+B)+\mathcal{O}(h^{p+1})$, for any such $\Psi_h^{[p,s]}$, solely in terms of commutators, whose exact solution $\varphi^{\tilde{H}}_h$ coincides with the result of the method and the discrepancy between this modified vector field and the original problem is called backward error. Notice that $\tilde H$ is sometimes called modified Hamiltonian and since it describes the flow of the numerical method, its energy is conserved.

Typically, the split is done such that $A=\tfrac12 \pvec^T\pvec$ and $B=B(t, \xvec, |\psi|)$ because then,
\[
	\varphi^{A}_h(\psi_0) = \mathcal{F}e^{-ih (k_x^2 + k_y^2)} \mathcal{F}^{-1} \psi_0,
\]
where $\mathcal{F}$ is the Fourier transform w.r.t. all spatial variables,
\[
	\mathcal{F}[\psi](k_x,k_y) = \frac{1}{\sqrt{2\pi}} \int_{-\pi}^\pi\int_{-\pi}^\pi \psi(x,y) e^{i(k_x x + k_y y)} dx dy.
\]
After spatial discretization, the integral can be evaluated using Fast-Fourier transforms and it remains to compute the exponential of a diagonal matrix constituted by the wave-numbers $-ih(k_x^2+k_y^2)$.

The flow of part $B$ can be computed as detailed in the following well-known lemma, 
\begin{lemma}\label{lemma}
		Let $B(t,  \xvec, |\psi|)$ be a real-valued function, then 
	\[
	i\dot \psi(t) = B(t,  \xvec, |\psi(t)|)\psi(t), \qquad \psi(0)=\psi_0
	\]
	has the solution
	\[
		\psi(t) = \exp\left(-i\int_0^t B(s,  \xvec, |\psi_0|) ds\right) \psi_0.
	\]
\end{lemma}
The proof relies on the simple calculation that $\frac{d}{dt}|\psi|^2 = \dot{\psi^*}\psi + {\psi^*}\dot\psi = 0,$
after plugging in the definition of the derivative. Then, the system reduces to a linear non-autonomous ODE.\hfill $\qed$

Suppose now, that the Hamiltonian has explicit time-dependencies, $H=H_A(t) + \eps B(t, \xvec,|\psi|)$.
Introducing time as two new coordinates $t_1, t_2$ makes the system treatable with splitting methods.
We write the Lie-derivative corresponding to the (nonlinear) vector field $H\psi$ as
\beq\label{eq.lie}
	\mathcal{L}_H = H_A(t_1)\left(\psi\frac{d}{d\psi}+\psi^*\frac{d}{d\psi^*}\right) + \frac{d}{dt_1}+ \eps B(t_2, \xvec,|\psi|)\left(\frac{d}{d\psi}+\frac{d}{d\psi^*}\right) + \frac{d}{dt_2}.
\eeq
The derivatives $d/dt_j$ are responsible for the evolution of the time-coordinates $t_j$.
Readers that are not familiar with Lie derivatives can use an analogy with a system of ODEs that is augmented by new time-coordinates,
\[
	\frac{d}{dt}y(t) = H_A(t_1)y(t)+\eps B(t_2)y(t), \quad
	\frac{d}{dt}t_1(t)=1, \quad
	\frac{d}{dt}t_2(t)=1, \qquad (y(0),t_1(0),t_2(0))=(y_0,0,0).	
	%\frac{d}{dt}\begin{pmatrix} y(t)\\ t_1(t) \\ t_2(t)\end{pmatrix}=
	%\begin{pmatrix} H_A(t_1)y(t)+\eps B(t_2)y(t)\\ 1 \\ 1\end{pmatrix}.
\]

There are several possibilities to split this enhanced operator (or system of equations): the computationally simplest pairs operators depending on $t_j$ with the evolution of $t_i$, $i\neq j$, e.g., 
$H_A(t_1)$ with $ d/dt_2$ and $\eps B(t_2)$ with $d/dt_1$. 
In this way, the splitting works exactly as for the autonomous situation since, in each internal step, the main operator $H_A(t_1)$ (or $B(t_2)$) is frozen and the other time-coordinate $t_2$ (or $t_1$) is advanced accordingly. %$F$ are frozen and the , with the only restriction that the new time-coordinates have to be updated in each internal splitting step.
Using the ODE analogy, this corresponds to a split into two systems
\[
	\frac{d}{dt}\begin{pmatrix} y(t)\\ t_1(t) \\ t_2(t)\end{pmatrix}=
	\begin{pmatrix} H_A(t_1)y(t)y(t)\\ 0 \\ 1\end{pmatrix},
	\quad
	\frac{d}{dt}\begin{pmatrix} y(t)\\ t_1(t) \\ t_2(t)\end{pmatrix}=
	\begin{pmatrix} \eps B(t_2)y(t)\\ 1 \\ 0\end{pmatrix}.
\]
A closer look at the error terms by computing the commutators reveals that this split, albeit simple, is by a factor $\eps$ less accurate since  the derivatives w.r.t $t_j$ mix large and small terms and we briefly examine how it can be recovered \cite{blanes10sac} after a short interlude on higher-order compositions in presence of a small parameter.

\paragraph{Near-integrable structure}
The appearance of the extra factor $\eps$ in the error terms is due to the separation of large and small parts in the splitting and extremely successful composition methods have been developed for this problem class \cite{mclachlan95cmi,blanes13nfo}.
The basic idea is to express the error in a power series in the time-step and in the small parameter,
\[
	\Phi_h-\varphi_h = \sum_{j\geq1}\sum_{k\geq s_j} e_{j,k}\eps^j h^{k+1} \quad \text{as}\  (h,\eps)\to(0,0),
\]
where the $s_j$ start from the first non-vanishing error coefficient and such a method $\Phi_h$ is said to be of \textit{generalized order} $(s_1,s_2,\ldots,s_m)$ (where $s_1\geq s_2\geq \cdots\geq s_m$).
Then, the coefficients $a_j,b_j$ of a splitting method \eqref{eq.splitting} can be chosen to construct methods of any generalized order, given that either of the two parts $A$ or $B$ is proportional to $\eps$.

This proportionality automatically carries over to the commutators $[A,B]\propto \eps$ but some additional consideration is required in the presence of time-dependencies since the operators causing the time-evolution are not necessarily proportional to $\eps$.
The simplest treatment makes use of the formulation in Lie derivatives \eqref{eq.lie}, for which we have already pointed out that freezing both parts will necessarily reduce the generalized order of a scheme.
On the upside, there is a remedy which has motivated this study: if the large part is advanced non-autonomously, the generalized error is preserved \cite{blanes10sac}.
In terms of Lie derivatives, this corresponds to taking the large part of \eqref{eq.lie} to be either
$A=H_A(t_1) + d/dt_1$ or $A=H_A(t_1) + d/dt_1+d/dt_2$. The latter option implies freezing the remainder $B$ and is usually preferred for simplicity and efficiency.
%\begin{align*}
%A=H_A(t_1) + \frac{d}{dt_1}, \ B = F(t_2) + \frac{d}{dt_2} \ , \text{or}\quad  
%A=H_A(t_1) + \frac{d}{dt_1}+\frac{d}{dt_2}, B=F(t_2),
%\end{align*}
%%i.e., $A=H_A(t_1) + d/dt_1$ (or $A=H_A(t_1) + d/dt_1+d/dt_2$, thus freezing part $B=F(t_2)$),
%\[%\beq\label{eq.lie2}
	%\underbrace{\overbrace{H_A(t_1)\left(\psi\frac{d}{d\psi}+\psi^*\frac{d}{d\psi^*}\right) + \frac{d}{dt_1}}^{A : \text{Both parts time-dependent}}+ \frac{d}{dt_2}}_{A : \text{Part B is frozen}} 
	%+ \eps F(t_2,x,|\psi|)\left(\frac{d}{d\psi}+\frac{d}{d\psi^*}\right)
%\]%\eeq
Our aim is to apply the highly efficient splitting methods for near-integrable systems \cite{mclachlan95cmi} to the problem at hand which is of similar structure. As discussed above, however, a proper application of splittings means 
to solve $\varphi^{H_A}_{t_j,t_j+a_jh}$ for a fractional time-step $a_jh$.

We stress that a propagation of $H_A$ in the autonomous (or frozen) setting using Laguerre-polynomials \cite{bao05afo,bao09agl,hofstaetter14cos} would be (at least) by a factor $\eps$ less accurate and the basis would have to be recomputed in each internal step. In consequence, the proposed algorithm which leaves the dominant part $\varphi^{H_A}_{t_j,t_j+a_jh}$ intact is the only way to preserve the generalized order of a given splitting method.

%%In the present context of vector field that consists of a dominant part $H_A$ plus some perturbation $\eps F$, 
%%.. we have a small factor.. want error in terms of $\eps h^p$
%\red{can write how it has to be done in order to not loose the small factor $\varepsilon$. and also how the others did it with the laguerre, also Diele}

\subsection{Time averaging}
From the considerations above, it is clear that instead of simply propagating frozen operators, we need to find good approximations to the exact flow $\varphi^{H_A}_{t_j,t_j+a_jh}$.
A cornerstone of the construction is the formal solution of a non-autonomous (linear) initial value problem, $\partial_t u(t)=A(t)u(t)$, in the form $u(t+h) = \exp(\Theta(t,t+h))u(t)$.
The \emph{Magnus expansion} \cite{magnus54ote} gives an expression for the exponent $\Theta$ using integrals of commutators of increasing length of the operator $A$ evaluated at different instances of time.
Its first two terms are 
\[
	\Theta(t,t+h) = \int_t^{t+h} A(s)ds + \frac12\int_t^{t+h}\int_t^{s_1}[A(s_1),A(s_2)]ds_2\,ds_1+\cdots,
\]
and recursive procedures exist to obtain higher-order corrections \cite{blanes09tme}.
The integrals can be efficiently computed by quadrature rules \cite{iserles99ots,blanes09tme} which will be exemplified in the numerical section.

It is easy to verify that the components of $H_A$ generate, via commutation, a ten-dimensional Lie algebra $\mathfrak{g}$ with basis 
\[
	\{x^2, p_x^2, y^2, p_y^2, xp_y, yp_x, xy, p_xp_y, xp_x+p_xx, yp_y+p_yy\},
\]
and since the Magnus expansion only operates by summation and commutation, we have $\Theta, \Theta^{[p]}\in\mathfrak{g}$ at any truncation order $p$, where the truncation is performed within the algebra, s.t., $\Theta^{[p]}=\Theta + \mathcal{O}(h^{p+1})$.
We stress that this yields a geometric integrator as staying in the correct algebra $\mathfrak{g}$ assures unitarity of the exponential.

Evaluating the commutators, the (truncated) Magnus expansion can be interpreted as averaged Hamiltonian, $\tilde{H}_{h}^{[p]}$,
%\begin{align}\nonumber
%\begin{multline}
\beq\label{magnus}
	\Theta_{\tn,\tn+h}^{[p]}= -ih\tilde{H}_{h}^{[p]} =-ih\,\Bigg(\,\frac{1}{2} (m_x p_x^2 + m_y p_y^2) + \frac{w_x}{2} x^2 + \frac{w_y}{2} y^2
										+\Omega_x xp_y - \Omega_y yp_x + \alpha xp_x +  \beta yp_y + \gamma xy + \delta p_xp_y
										\Bigg)\in \mathfrak{g},
%	k_1 x^2 + k_2 p_x^2 + k_3 y^2 + k_4 p_y^2 + k_5 xp_y + k_6 yp_x + k_7 xy + k_8 p_xp_y + e_9 xp_x + e_{10} yp_y
%\end{align}
\eeq
%\end{multline}
for some coefficients $m_x, m_y,w_x,$ etc. that depend on $h$, $p$ and on the integrals of $\omega_x(s)^2, \omega_y(s)^2$ over the interval $[t,t+h]$.
Fixing the time-step $h$, it is clear that the flow $e^{-it\tilde{H}_{h}^{[p]}}$ of
\[
	i\partial_t \psi(\xvec, t) = \tilde{H}_{h}^{[p]}\psi(\xvec, t)
\]
coincides with the truncated Magnus expansion $\exp(\Theta^{[p]}_{\tn,\tn+h})$ at $t=h$.
In principle, one could think about using some standard splitting method for this operator, but the mixed terms $xp_x,yp_y$ cannot be diagonalized by means of Fourier transforms.
Notice the relationship between the decomposition \eqref{eq:decomp} and commutator-free Magnus methods \cite{blanes06fas} since the computationally difficult terms arise after commutation only. These methods, however, require a higher number of FFTs (double for order 4).

\subsection{The decomposition method}
After having obtained a $p$th-order approximation to the averaged Hamiltonian and thus $\Theta^{[p]}$ for a time-step $h$, we will now show how to accurately compute $e^{\Theta^{[p]}}$ without having to evaluate mixed operators.
The key to our endeavor is the finiteness of the underlying algebra to which any such $\Theta$ belongs.
\begin{theorem}
Let $\psi(x,y)$ be a  sufficiently smooth wave function\footnote{The smoothness is a generic requirement to ensure efficiency of Fourier methods for the approximation of the derivatives, e.g., the Laplacian.}, then, for sufficiently small $h>0$ and $\Theta^{[p]}_{t,t+h}$ from \eqref{magnus}, there exist scalars $e_j,f_k,g_l$ such that
\beq\label{eq:theorem}
	e^{f_0 x^2}
	e^{f_1 y^2 + g_1 p_x^2 - e_1 yp_x}
	e^{f_2 x^2 + g_2 p_y^2 + e_2 xp_y}
	e^{f_3 y^2 + g_3 p_x^2 - e_3 yp_x}
	\,\psi(x,y)
	= e^{\Theta^{[p]}_{t,t+h}}\,\psi(x,y).
\eeq
\end{theorem}
\begin{proof}
Together with the BCH formula \eqref{eq.bch}, we deduce that for sufficiently small $h$, the left-hand side of \eqref{eq:theorem} is summable and can be expressed as a single exponential of a linear combination of the basis elements \cite{wei63las,wilcox67jmp}.
The scalar coefficients in \eqref{eq:theorem} are determined by equating the resulting exponent to the averaged Hamiltonian $-ih\tilde{H}_{h}^{[p]}$ in \eqref{magnus}. 
The missing operators $xp_x, yp_y$, can be generated by the following commutators,
\[
	[x^2, p_x^2],\ [y^2, p_y^2],\ [p_xp_y, xy], \ [xp_y, yp_x],
\]
only. 
Our ansatz hence includes two free variables multiplying 
$x^2$ ($f_0, f_2$), $y^2$ ($f_1, f_3$), $p_x^2$ ($g_1,g_3$) and $yp_x$ ($e_1, e_3$). 
One in each pair will satisfy the equation for the corresponding basis element, whereas the other can be used to create the terms $xy, p_xp_y, xp_x$ and $yp_y$.
\end{proof}
The BCH formula, however, is of very limited use for the actual computation of this equation since the number of appearing commutators grows exponentially with the order. 

Instead, we propose an alternative procedure to derive the coefficients $e_j,f_k,g_l$, which extends results from Ref.~\cite{bader11fmf} and relies on finding a faithful (injective) representation of the operator Lie algebra $\mathfrak{g}$.
This Lie algebra isomorphism drastically simplifies all calculations since we will be able to verify the decomposition \eqref{eq:decomp} by computations in a (low-dimensional) matrix setting.
The correspondence principle $i[\cdot,\,\cdot]\to \{\cdot,\,\cdot\}$ between the quantum Lie bracket and the Poisson bracket gives an elegant method to find the isomorphism by considering the equivalent classical Hamiltonian system for $\tilde{H}(x,y,p_x,p_y)$.
For convenience of the reader, we recall that the classical equations of motion are given by
\[
\dot{\mathbf{r}} = \frac{\partial H}{\partial \mathbf{p}}, \qquad 
\dot{\mathbf{p}} = -\frac{\partial H}{\partial \mathbf{r}},
\]
which translates for our Hamiltonian $\tilde{H}_h^{[p]}$ from \eqref{magnus} to 
%
%\red{maybe write out classical hamiltons equations}
\beq\label{eq.classical}
	\frac{d}{dt}\begin{pmatrix}
	x \\ y \\ p_x \\ p_y
	\end{pmatrix}
	=
	 M\begin{pmatrix}
	x \\ y \\ p_x \\ p_y
	\end{pmatrix}
	= 
	\begin{pmatrix}
		\alpha   & -\Omega_y & m_x 		 & \delta\\
		\Omega_x & \beta & \delta & m_y\\
		-w_x& -\gamma   &-\alpha& -\Omega_x\\
		-\gamma & -w_y &\Omega_y& - \beta
	\end{pmatrix}
	\begin{pmatrix}
	x \\ y \\ p_x \\ p_y
	\end{pmatrix} 
	.
\eeq
The injectivity is a consequence of the uniqueness of each matrix element w.r.t. the coefficients in $\tilde{H}$ and it is easy to verify that the matrices indeed form an isomorphic Lie algebra with the standard matrix commutator.

Although it is too cumbersome to evaluate the solution operator $\exp(hM)$ in closed form, it is a straightforward numerical task.

Next, we take a look at the left-hand side of \eqref{eq:decomp}, and for illustration, we compute the rightmost exponent,
$\exp({f_3 y^2 + g_3 p_x^2 - e_3 yp_x})$ in the matrix algebra which solves the equation
\[
	\frac{d}{dt}\begin{pmatrix}
x \\ y \\ p_x \\ p_y
\end{pmatrix}
= 
\begin{pmatrix}
	0 & -e_3 & 2g_3 & 0\\
	0 & 0 & 0 & 0\\
	0& 0 & 0 & 0\\
	0& -2f_3 & e_3 & 0
\end{pmatrix}
\begin{pmatrix}
x \\ y \\ p_x \\ p_y
\end{pmatrix} = N\begin{pmatrix}
x \\ y \\ p_x \\ p_y
\end{pmatrix}, 
\]
and hence $\exp{(f_2 x^2 + g_2 p_y^2 + e_2 xp_y)}=\exp(N)=1+N$.
The remaining matrices in the exponents of $\Phi_{\tn,\tn+h}^{[p]}$ are also nilpotent and can be trivially exponentiated,
\def\Id{1}
%\begin{align*}
%e^{f_1 y^2 + g_1 p_x^2 - e_1 yp_x}&=
%\begin{pmatrix}
	%1 & -e_1 & 2g_1 & 0\\
	%0 & 1 & 0 & 0\\
	%0& 0 & 1 & 0\\
	%0& -2f_1 & e_1 & 1
%\end{pmatrix},
%e^{f_2 x^2 + g_2 p_y^2 + e_2 xp_y}=\begin{pmatrix}
	%1   & 0 & 0 & 0\\
	%e_2 & 1 & 0 & 2g_2\\
	%-2f_2& 0 & 1 & -e_2\\
	%0   & 0 & 0 & 1
%\end{pmatrix},\\
%e^{f_0 x^2}& = \begin{pmatrix}
	%1   & 0 & 0 & 0\\
	%0 & 1 & 0 & 0\\
	%-2f_0& 0 & 1 & 0\\
	%0   & 0 & 0 & 1
%\end{pmatrix}.
%\end{align*}
%
%\begin{multline*}
%e^{f_0 x^2}
%e^{f_1 y^2 + g_1 p_x^2 - e_1 yp_x}
%e^{f_2 x^2 + g_2 p_y^2 + e_2 xp_y}
%e^{f_3 y^2 + g_3 p_x^2 - e_3 yp_x}=\\
%\begin{pmatrix}
	%1 & -e_3 & 2g_3 & 0\\
	%0 & 1 & 0 & 0\\
	%0& 0 & 1 & 0\\
	%0& -2f_3 & e_3 & 1
%\end{pmatrix}
%\begin{pmatrix}
	%1   & 0 & 0 & 0\\
	%e_2 & 1 & 0 & 2g_2\\
	%-2f_2& 0 & 1 & -e_2\\
	%0   & 0 & 0 & 1
%\end{pmatrix}
%\begin{pmatrix}
	%1 & -e_1 & 2g_1 & 0\\
	%0 & 1 & 0 & 0\\
	%0& 0 & 1 & 0\\
	%0& -2f_1 & e_1 & 1
%\end{pmatrix}
%\begin{pmatrix}
	%1   & 0 & 0 & 0\\
	%0 & 1 & 0 & 0\\
	%-2f_0& 0 & 1 & 0\\
	%0   & 0 & 0 & 1
%\end{pmatrix}
%\end{multline*}
%\red{Have put the factor 1/2 in the exponents to save space}
\begin{multline}\label{eq.fullmatrices}
e^{f_0 x^2/2}
e^{f_1 y^2/2 + g_1 p_x^2/2 - e_1 yp_x}
e^{f_2 x^2/2 + g_2 p_y^2/2 + e_2 xp_y}
e^{f_3 y^2/2 + g_3 p_x^2/2 - e_3 yp_x}=\\
\begin{pmatrix}
	1 & -e_3 & g_3 & 0\\
	0 & 1 & 0 & 0\\
	0& 0 & 1 & 0\\
	0& -f_3 & e_3 & 1
\end{pmatrix}
\begin{pmatrix}
	1   & 0 & 0 & 0\\
	e_2 & 1 & 0 & g_2\\
	-f_2& 0 & 1 & -e_2\\
	0   & 0 & 0 & 1
\end{pmatrix}
\begin{pmatrix}
	1 & -e_1 & g_1 & 0\\
	0 & 1 & 0 & 0\\
	0& 0 & 1 & 0\\
	0& -f_1 & e_1 & 1
\end{pmatrix}
\begin{pmatrix}
	1   & 0 & 0 & 0\\
	0 & 1 & 0 & 0\\
	-f_0& 0 & 1 & 0\\
	0   & 0 & 0 & 1
\end{pmatrix}.
\end{multline}
%\red{
Notice that we have changed the exponents by introducing factors $1/2$ in the previous equation to get a more compact notation. Later on, we revert to the original form for better readability.
Furthermore, the multiplication order of the matrices had to be reversed w.r.t. the exponentials to account for their nature as Lie derivatives (cf. \textit{Vertauschungssatz} \cite{hairer06gni}) and equality holds if both sides are understood as flows.
%}

Multiplication of the matrix exponentials \eqref{eq.fullmatrices} according to \eqref{eq:decomp} yields a $4\times4$ matrix of multivariate polynomials of maximum degree four that has to be equated to $\exp(hM)$.
The system $u'=Mu$, $u=(x,y,p_x,p_y)^T$ has ten degrees of freedom, originating from the linearly independent basis terms and the same number of variables has been introduced in \eqref{eq:decomp}.
It is clear that if we had allowed the appearance of the mixed terms $xp_x, yp_y$, the composition could be solved easily since we have a free variable for each basis element \cite{wei63las,wilcox67jmp}.

With the aim of obtaining a \emph{Fourier-diagonalizable} decomposition, i.e., terms that are diagonal after Fourier transform, we examine the structure coefficients of the algebra.
It turns out that $xp_x, yp_y$, can be generated by the following commutators,
\[
	[x^2, p_x^2],\ [y^2, p_y^2],\ [p_xp_y, xy], \ [xp_y, yp_x],
\]
only. 
Our ansatz hence includes two free variables multiplying 
$x^2$ ($f_0, f_2$), $y^2$ ($f_1, f_3$), $p_x^2$ ($g_1,g_3$) and $yp_x$ ($e_1, e_3$). 
One in each pair will satisfy the equation for the corresponding basis element, whereas the other can be used to create the terms $xy, p_xp_y, xp_x$ and $yp_y$.
% under the fairly general assumption $m_x,m_y, \omega_x,\omega_y, \Omega_x,\Omega_y \neq0$.
We deduce that all basis terms can be generated and thus, for sufficiently small $h$, a solution of the only formally overdetermined $4\times4$ nonlinear algebraic system can be computed, e.g., with the Gauss-Newton algorithm.

With the help of a Gr\"obner basis and simple algebra, the number of variables can be reduced to accelerate the algorithm.

Notice that there are multiple choices of possible compositions.
For example, at the same cost, we could have replaced the outer exponential by $e^{f_0 xy}$ or introduced the terms $e^{kp_xp_y}$ before and after the center exponential (setting $f_0=0$).
It is not clear whether other choices for the decomposition are more advantageous.

For simplicity of the presentation, we have chosen a simple form of the Hamiltonian \eqref{eq:1}, but using our methodology, analogous methods can be derived for virtually all\footnote{Excluding certain pathological cases, e.g., absence of kinetic energy etc.} relevant polynomial Hamiltonians of degree $\leq2$ in any dimension with arbitrary time-dependencies.

A full time-step of the algorithm is summarized in Table~\ref{tab.algo}.
\begin{table}
\begin{tabular}{p{5mm}p{.9\textwidth}}
\toprule
\multicolumn{2}{l}{\textbf{Algorithm}}\\[2mm]
1 & Compute the Magnus expansion \eqref{magnus} of $H_A$ up to the desired order $p$.
			We refer to the review \cite{blanes09tme} for the expansion algorithm as well as for explicit formulae for order four and six methods \cite[pp. 205-206]{blanes09tme}.\\[1mm]
2 & Rewrite the resulting commutators in the basis of the Lie algebra to obtain the modified Hamiltonian \eqref{magnus}.\\[1mm]
3 & 
Solve the resulting (small) polynomial system $\exp(M)=E$, where $M$ is defined in \eqref{eq.classical} and $E$ is the right-hand-side of \eqref{eq.fullmatrices}
to obtain the coefficients $e_j, f_k, g_l$.\\[1mm]
4 & Apply the composition $\Phi_{t,t+h}^{[p]}$ \eqref{eq:decomp} using six one-dimensional FFTs $\mathcal{F}_x,\mathcal{F}_y$ to diagonalize the exponentials,\\[2mm]
&\multicolumn{1}{c}{
$\displaystyle 
\Phi_{t,t+h}^{[p]}(\psi(x,y))= 
	e^{f_0 x^2}
	\mathcal{F}_x^{-1}
	e^{f_1 y^2 + g_1 \hat p_x^2 - e_1 y\hat p_x}
	\mathcal{F}_x\mathcal{F}_y^{-1}
	e^{f_2 x^2 + g_2 \hat p_y^2 + e_2 x\hat p_y}
	\mathcal{F}_y\mathcal{F}_x^{-1}
	e^{f_3 y^2 + g_3 \hat p_x^2 - e_3 y \hat p_x}
	\mathcal{F}_x\ \psi(x,y)$}\\[2mm]
	&Note that the momentum operators $p_k$ have been replaced by $\hat p_k$ to indicate that they correspond to their (diagonal) representation in the momentum space.
	\\[2mm] 	
\bottomrule
\end{tabular}
\caption{\label{tab.algo} Algorithm for the computation of $\varphi_{t,t+h}^{H_A}$. Expressions for steps 1 and 2 can be precomputed and only need to be updated with the current values of the time coordinate. The steps can be followed at a concrete example in section~\ref{sec.num}.}
\end{table}
It is worth pointing out that the extra effort is virtually independent of the order choose for the Magnus expansion and can be neglected as the number of grid points increase. 

In total, one step of the algorithm requires the application of two 1D and two 2D Fourier transforms, which can be implemented at the cost of three 2D-FFTs, and prior to evolving the wave function, the coefficients are determined through exponentiating a $4\times 4$ matrix and solving a small nonlinear system.
The effort for the solution of the (formally) overdetermined system, which can be done by a least-square algorithm, is marginal since -- for small time-steps -- the solution is not far from $0\in\R^{10}$.

\subsubsection{Special cases}
In passing, we mention some further cases, for which the algebra simplifies. 
\paragraph{Isotropic trap}
			\[
				H = \frac{1}{2m(t)}\left(p_x^2 + p_y^2\right) + \frac{1}{2}m(t)\omega(t)^2 (x^2 + y^2) + \Omega(t) L_z
			\]
			Due to cancellations, the commutators of $H$ at different instances lie in the span of 
			\[
				\left\{p_x^2+p_y^2, x^2+y^2, L_z, (xp_x+p_xx)+(yp_y+p_yy)\right\}
				\]
				and any Magnus integrator can be written as effective Hamiltonian 
			\[
				\tilde{H}_{t,t+h} = a_{t,t+h}(p_x^2 + p_y^2) + b_{t,t+h} (x^2 + y^2) + c_{t,t+h} L_z
														+ d_{t,t+h} (xp_x + yp_y) + e_{t,t+h} xy.
			\]
			
\paragraph{Linear interaction} For time-dependencies proportional to the linear components only,
			\[
				H = \frac{1}{2}\left(p_x^2 + p_y^2\right) + \frac{\omega_0^2}{2}\left(x^2+y^2\right) + \xi_x(t)x + \xi_y(t)y + \Omega L_z,
			\]
			the following terms in algebra do not appear: $xy, p_xp_y, xp_x, yp_y$.
			It is therefore sufficient to employ a symmetric composition including the linear terms in the exponent,
			\[
	e^{f_1 y^2 + g_1 p_x^2 - e_1 yp_x + d_1 y + c_1 p_x }
	e^{f_2 x^2 + g_2 p_y^2 + e_2 xp_y + d_2 x + c_2 p_y }
	e^{f_1 y^2 + g_1 p_x^2 - e_1 yp_x + d_1 y + c_1 p_x }.
			\]
			The equations for the parameters have to be obtained in a slightly different way which is described below, preceding eqns. \eqref{eq.examplelin}.
		
\paragraph{General quadratics}
For more complicated Hamiltonians, 
\beq\label{eq:rot:fullE15}
		\tilde{H}_h^{[p]} = \sum_{j=1}^{15} \alpha_j(h) E_j,
\eeq
in the algebra $\mathfrak{g}$ with basis 
\begin{align}\nonumber
	E_1 &= x,			&		
	E_2 &= p_x,		& 
	E_3 &= \tfrac12 x^2, &  
	E_4 &= \tfrac12 p_x^2,  & 	
	E_5 &= \tfrac12\left(xp_x+p_xx\right),\\ \nonumber
%	 & 		& 		
	E_6 &= y,			&		
	E_7 &= p_y,		& 
	E_8 &= \tfrac12 y^2,&
  E_{9}&= \tfrac12 p_y^2,  &
	E_{10} &= \tfrac12\left(yp_y+p_yy\right),\\ \nonumber%\label{eq:2:basisdef}
%	 & 		& 
	E_{11} &= xy, 		& 		
	E_{12} &= p_xp_y,	& 		
	E_{13} &= xp_y,		& 
	E_{14} &= yp_x,		& 
	E_{15} &= 1, 		& 	
\end{align}
in particular when linear terms are involved, the described procedure fails to a certain degree: As for the one dimensional harmonic oscillator problem \cite{bader11fmf}, the phase relation  cannot be recovered since the classical mechanical equivalent does not enter the equations of motion.
In principle, one could approximate the phase numerically, by introducing a new variable proportional to the phase $E_{15}$ and derive a system of differential equations for all parameters by interpreting them as time-dependent functions and plugging the ansatz into the Schr\"odinger equation with Hamiltonian \eqref{eq:rot:fullE15} after Magnus averaging.
Then, the resulting scalar functions are evaluated using the presented algorithm on a fixed grid which partitions the time-step interval $[t_n, t_n+h]$ and then we finally solve the differential equation for the free phase parameter numerically, or alternatively, we make use of the BCH formula.

This effort can be spared since only the global phase information is lost which is not observable.
The polynomial system to be solved then needs additional degrees of freedom to cater for the linear contributions and to close the discussion, we conjecture that there exist (under mild assumptions\footnote{In order to recover the mixed terms $xp_x, yp_y$, a pair of possible generators must be present in the Hamiltonian, e.g., $x,p_x^2$ can generate $xp_x$ through commutation.}) imaginary coefficients, such that
\beq\label{eq.decompfull}
\Psi_h=
	e^{n_1 x^2 + m_1 x}
	e^{f_1 y^2 + g_1 p_x^2 - e_1 yp_x + k_1 p_x}
	e^{f_2 x^2 + g_2 p_y^2 + e_2 xp_y + k_2 p_y}
	e^{f_3 y^2 + g_3 p_x^2 - e_3 yp_x + k_3 p_x}
	e^{n_2 x^2 + m_2 x},
\eeq
is the solution of the SE with Hamiltonian \eqref{eq:rot:fullE15} for small values of the $\alpha_j$.
In the remainder of this section, we compute the system of equations which will determine the scalar coefficients in the exponents.
Note that a slightly different methodology has to be applied to account adequately for the linear terms.
The corresponding classical mechanical system is 
\[
	\frac{d}{dt}\begin{pmatrix}
	x \\ y \\ p_x \\ p_y
	\end{pmatrix}
	=\begin{pmatrix}
	\nabla_{p_x}  \\ \nabla_{p_y} \\ -\nabla_{x} \\ -\nabla_{y} 
	\end{pmatrix}\tilde{H}_h^{[p]}
	=
	\begin{pmatrix}
		\alpha_2 + \alpha_4 p_x + \alpha_5 x + \alpha_{12} p_y + \alpha_{14}y\\
		\alpha_7 + \alpha_9 p_y + \alpha_{10} y + \alpha_{12} p_x + \alpha_{13}x\\
		-(\alpha_1 + \alpha_3 x + \alpha_5 p_x + \alpha_{11} y + \alpha_{13}p_y)\\
		-(\alpha_6 + \alpha_8 y + \alpha_{10} p_x + \alpha_{11} x + \alpha_{14}p_x)\\
	\end{pmatrix}
	=
	\underbrace{\begin{pmatrix}
		\alpha_5 & \alpha_{14} & \alpha_4 & \alpha_{12}\\
		\alpha_{13} & \alpha_{10} & \alpha_{12} & \alpha_{9}\\
		-\alpha_{3} & -\alpha_{11} & -\alpha_{5} & -\alpha_{13}\\
		-\alpha_{11} & -\alpha_{8} & -\alpha_{10} & -\alpha_{14}\\
		\end{pmatrix}}_{=M}
	\begin{pmatrix}
	x \\ y \\ p_x \\ p_y
	\end{pmatrix} +
	\begin{pmatrix}
	\alpha_2 \\ \alpha_7 \\ -\alpha_1 \\ -\alpha_6
	\end{pmatrix} 
	.
\]
The exact solution is obtained using the variation-of-constants formula which handles the linear contributions in the Hamiltonian,
\beq\label{eq.fullsol}
(x(t),y(t),p_x(t),p_y(t))^T = e^{tM}\ (x(0),y(0),p_x(0),p_y(0))^T\ + \ M^{-1}\left(e^{tM}-1\right)\ (\alpha_2 ,\alpha_7 ,-\alpha_1 ,-\alpha_6)^T
\eeq
Similarly, the decomposition cannot be written anymore as a product of matrix exponentials, instead, the exponentials are interpreted as flows and computed accordingly.
For clarity, we describe how to compute the flows in \eqref{eq.decompfull}. Take, e.g., the Hamiltonian $H=n_1x^2+m_1x$ with corresponding flow $e^{n_1x^2+m_1x}$.  Then, it is trivial to compute its action on some initial value $(x,y,p_x,p_y)^T$,
\beq\label{eq.examplelin}
\begin{split}
	e^{n_1x^2+m_1x}(x,y,p_x,p_y)^T &= (x,y,p_x - (2n_1x+m_1),p_y)^T,\\
	e^{f_1 y^2 + g_1 p_x^2 - e_1 yp_x + k_1 p_x}(x,y,p_x,p_y)^T&= (x + 2g_1 p_x - e_1y + k_1,
																y,p_x ,p_y-(2f_1 y-e_1p_x))^T,\\
	e^{f_2 x^2 + g_2 p_y^2 + e_2 xp_y + k_2 p_y}(x,y,p_x,p_y)^T &= (x,y+(2g_2p_y+e_2x+k_2),p_x -(2f_2 x+e_2 p_y),p_y)^T,\\
	e^{f_3 y^2 + g_3 p_x^2 - e_3 yp_x + k_3 p_x}(x,y,p_x,p_y)^T &= 
			(x + 2g_3 p_x - e_3y + k_3,y,p_x ,p_y-(2f_3 y-e_3p_x))^T,\\
	e^{n_2 x^2 + m_2 x}(x,y,p_x,p_y)^T &= (x,y,p_x - (2n_2x+m_2),p_y)^T.
\end{split}
\eeq
Composing in the order of \eqref{eq.decompfull}, which means from bottom to top, and writing the result as a affine system in the initial values, we equate with  \eqref{eq.fullsol} to get
\beq
e^{tM}\ (x,y,p_x,p_y)^T\ + \ M^{-1}\left(e^{tM}-1\right)\ (\alpha_2 ,\alpha_7 ,-\alpha_1 ,-\alpha_6)^T
= N (x,y,p_x,p_y)^T + (a,b,c,d)^T.
\eeq
Since this equality needs to hold for all initial values $(x,y,p_x,p_y)$, we can read off the equations separately from the homogeneous and inhomogeneous parts.

\subsection{Higher dimensions}
It is straightforward to generalize the results to arbitrary spatial dimensions $n$ given that the potentials remain quadratic. 
The only noteworthy detail is that the dimensions of the matrices that yield the polynomial system scale with $2n\times 2n$.
In the particular case of a harmonic trap in the $z$-axis -- ceteris paribus --
the Hamiltonian can be written as a sum 
\[
H = H_A + H_z + G(x,y,z),
\]
where $H_z = \frac{1}{2} p_z^2 + \frac{1}{2}\omega_z^2 z^2$.
A sensible splitting groups commuting terms $A=H_A+H_z$, leaving the (small) remainder $B=G$ in order to compute one splitting step
\[
	e^{-ih a_j A}e^{-ih b_jB} = e^{-ih a_j H_A}e^{-ih a_j H_z}e^{-ih b_j G},
\]
and $e^{-ihH_z}$ can be solved using two FFTs in the $z$-direction since
\[
	e^{-ihH_z}=e^{-i\tan(h\omega_z/2)z^2/2}e^{-i\sin(h\omega_z)p_z^2/2}e^{-i\tan(h\omega_z/2)z^2/2},
\]
 see \cite{bader11fmf,chin05foa}.

\section{Numerical results}\label{sec.num}
In this section, we will illustrate the performance of the algorithm in two settings.
First, we consider the plain decomposition method where we expect to recover the order of the underlying Magnus expansion. Second, we add a small nonlinearity $g$ and embedd the rotation Hamiltonian $H_A$ into a second order Strang splitting in order to illustrate the recovery of the generalized order (4,2):
Even though the method is of formal order 2, its error is proportional to size of the (nonlinear) perturbation. 
Embedding the decomposition in a higher order splitting method would maintain the full order while keeping the near-integrable structure.

To numerically verify the proposed algorithm, we choose the Hamiltonian
\beq\label{eq:rot:example}
	H_A(t) = \frac{1}{2} \left(p_x^2+ p_y^2\right) + \frac{1}{2} \left(\omega_x(t)^2 x^2 + \omega_y(t)^2y^2\right) + \Omega L_z,
\eeq
where $\omega_x(t)^2=\omega_0^2(1+\sin(t/2)), \omega_y(t)^2=(\omega_0^2-\sin(t/2))$ and $\omega_0^2=4, \ \Omega={1/10}$.
The spatial domain is discretized with $128\times 128$ grid points on $[-10, 10]^2$ and we integrate the normalized initial condition $\psi_0\propto(x+\ii y)e^{-(x^2+y^2)/2}$ until the final time $T=3$.
The first two steps are numbered as in Table~\ref{tab.algo}.\\
\emph{Step {1}:} For the time-averaging we choose a fourth-order Magnus integrator that is in turn based on the fourth-order Gauss-Legendre quadrature,
\beqs
 \Theta^{[4]}_{t,t+h}= -\ii\frac{h}{2}\left(H(t_1)+H(t_2)\right) + \frac{h^2}{4\sqrt{3}} [-\ii H(t_1), -\ii H(t_2)],
\eeqs
where $t_j=t+hc_j$ with the standard Gauss-nodes $c_{1,2}=(1 \mp 1/{\sqrt{3}})/2$.\\
\emph{Step {2}:} Evaluating the commutator leads to the averaged Hamiltonian
\beqs
\ii \Theta^{[4]}_{t,t+h}=
\!%
\begin{multlined}[t][.8\displaywidth]
\frac{h}{2} \left(p_x^2+ p_y^2\right) 
+ \frac{h}{2} \frac{\omega_x(t_1)^2+\omega_x(t_2)^2}{2} x^2 
+ \frac{h}{2} \frac{\omega_y(t_1)^2+\omega_y(t_2)^2}{2} y^2 
+ h \Omega L_z\\
+ \frac{h^2}{4\sqrt{3}} \left( \frac{xp_x+p_xx}{2} \left(\omega_x(t_2)^2-\omega_x(t_1)^2\right)
														+\frac{yp_y+p_yy}{2} \left(\omega_y(t_2)^2-\omega_y(t_1)^2\right) \right)\\
+ \frac{h^2}{4\sqrt{3}} \, \Omega \left(  \left(\omega_y(t_2)^2-\omega_y(t_1)^2\right) - \left(\omega_x(t_2)^2-\omega_x(t_1)^2\right) \right) xy .
\end{multlined}
\eeqs
\subsection{Rotating linear Hamiltonian}
In a first experiment, we compare the split \eqref{eq:decomp} against a standard symmetric approach which uses the same number of FFTs per step:
\beq\label{eq:rot:std}
\Psi^{[2]}_{t,t+h}= 
%\exp\left(
e^{-\ii \frac{h}{2} \omega_x^2(t) x^2}%\right)
e^{-\ii \frac{h}{2} \left(p_x^2/2 - \Omega yp_x\right)}
e^{-\ii h \left(p_y^2/2 + \Omega xp_y\right)}
e^{-\ii \frac{h}{2} \left(p_x^2/2 - \Omega yp_x\right)}
%\exp\left(
e^{-\ii \frac{h}{2} \omega_x^2(t+h) x^2}.%\right).
\eeq
The results are shown in Fig.~\ref{fig1} and clearly show the correct order and high accuracy of the new method.
\begin{figure}[!ht]%
\centering
\parbox{\textwidth}{\includegraphics[width=\textwidth]{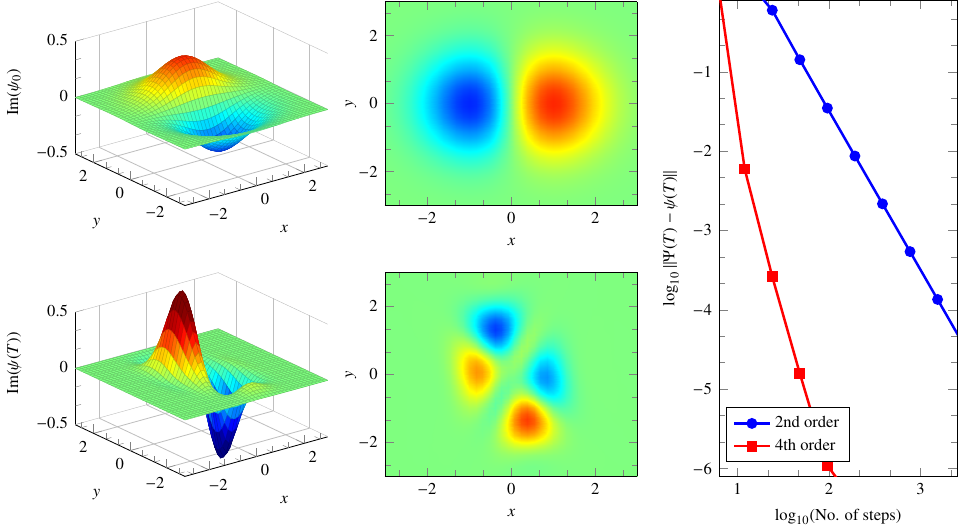}}
\caption{\label{fig1} (color online) The rightmost column shows the efficiency curves for the 2D rotating harmonic oscillator \eqref{eq:rot:example} integrated using $N_x=N_y=128$ equidistant grid points on $[-10,10]^2$.  
In the first row, the initial condition with imaginary (left) and real part (center) are displayed, whereas the evolution at time $T=3$ is depicted in the second row, both for imaginary (left) and real part (center). The real part is shown from above and the colormaps are kept constant in each panel, i.e., the same color corresponds to the same value. 
}
\end{figure}

\subsection{Rotating BEC with weak nonlinearity}
\label{sec.num.weaknonlin}
In a second experiment, the Hamiltonian $H_A$ is perturbed by a cubic nonlinearity,
\beq\label{eq:rot_pert_h}
		H(t) = H_A(t) + g |\psi|^2,
\eeq
with $g=1$ and the experimental setup is taken as above with $\omega_0^2=2,\ \Omega=1/5$ and $N_x=N_y=256$ grid points on the mesh $[-15,15]^2$. Instead of embedding our method within a higher-order splitting, we chose a simple Strang-type approach where the perturbation $H-H_A$ is appended on both sides of the method, 
\beq
\label{eq.nonlin2nd}
e^{-i\frac{h}{2}g|\psi|^2} \Psi	e^{-i\frac{h}{2}g|\psi|^2}.
	%,\qquad	e^{-i\frac{h}{2}g|\psi|^2} \Phi^{[4]}_{t,t+h}	e^{-i\frac{h}{2}g|\psi|^2}.
\eeq
We denote by ROT(2) the second order method from our construction with $\Psi=\Phi^{[4]}_{t,t+h}$ as in \eqref{strang} and by STD(2) the symmetric approach $\Psi=\Psi^{[2]}_{t,t+h}$ from \eqref{strang:classic}, respectively.
Of course, despite the fourth-order Magnus expansion, we only expect a second-order integrator, however, with much smaller error terms when compared to \eqref{eq:rot:std} due to the smallness of the perturbation. The method, as well as the method of generalized order (4,2) from \cite{mclachlan95cmi} are expected to perform well for large step-sizes and small nonlinearities.
At higher order, the benefits of our method are even more pronounced: Taking into account that we are facing a multi-component splitting with explicit time-dependencies, it is not trivial to derive an efficient splitting algorithm.
The usual approach is to design a symmetric second order method $\Psi_h$, e.g., \eqref{strang:classic} and then compose it with itself using fractional time-steps as in
$$
	\Psi_{\gamma h}\circ\Psi_{(1-2\gamma)h}\circ \Psi_{\gamma h},
$$
where $\gamma=1/(2-2^{1/3})$. This procedure is the well-known as triple jump \cite{suzuki90fdo,yoshida90coh} and the resulting fourth-order method Y(4) will be used for reference in the experiment.
In our design, the flow of the linear (in the wave function) and explicit Hamiltonian are efficiently computed together and thus, the split only contains two components $A,B$. We can thus use standard fourth-order methods such as the optimized SRKN$_6^b$ (BM(4)) from \cite{blanes02psp} to illustrate the improved performance in comparison with the triple-jump.

The results in Fig.~\ref{fig2} confirm the predicted behavior.
At any given precision, roughly half of the steps are needed by the new algorithm using the same number of FFTs.
As the perturbation, in this case the nonlinearity, becomes weaker, the benefits increase - on the other hand, when the remainder cannot be considered to be a perturbation, the second-order curves will get closer and almost coincide.
\subsection{Rotating BEC with strong nonlinearity}
Using the setup from the previous experiment but with a strong interaction parameter $g=50$, the method maintains its advantages at higher precision as expected since it allows the use of optimized splitting methods in contrast to the simple Yoshida composition for higher order.
The results are shown in Fig.~\ref{fig3}. At lower precision, both split variants perform equally. However, when higher precision with higher order is sought, our methodology maintains the advantage since optimized two-component splittings, e.g., from \cite{blanes02psp}, can be used.

\begin{figure}[!ht]%
\centering
\parbox{\textwidth}{\includegraphics[width=\textwidth]{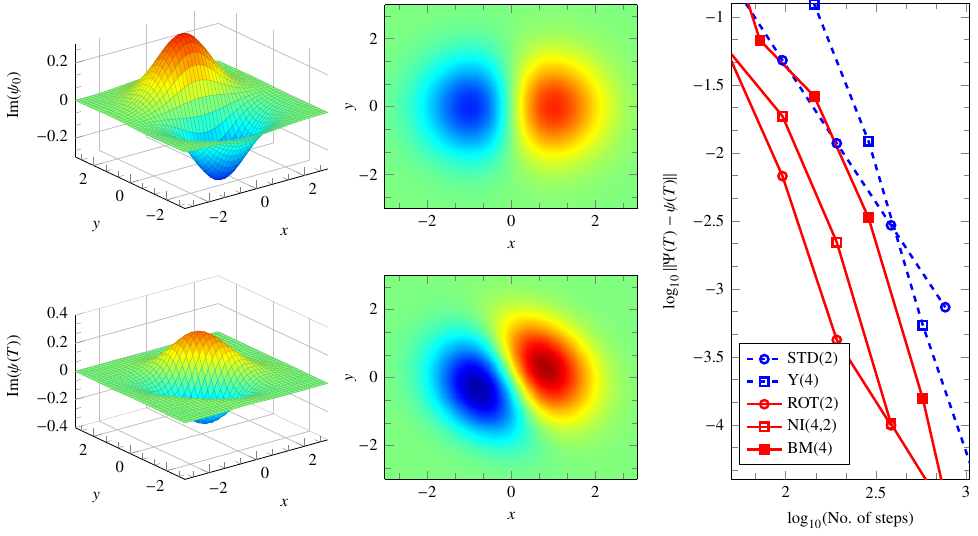}}
\caption{\label{fig2} (color online) For detailed captions, cf. Fig.~\ref{fig1}. Results for the Hamiltonian \eqref{eq:rot_pert_h} with a weak nonlinearity $g=1$. The top row shows the imaginary (left) and real part (center) of the initial condition, whereas the corresponding pictures for the exact solution at $T=5$ are displayed in the bottom row. The right panel demonstrates the smaller error constant for the proposed decomposition (red solid) in comparison with the standard split (blue dashed). The circles correspond to second order methods ROT(2) from \eqref{eq.nonlin2nd} and STD(2) \eqref{strang:classic}.}
\end{figure}

\begin{figure}[!ht]%
\centering
\parbox{\textwidth}{\includegraphics[width=\textwidth]{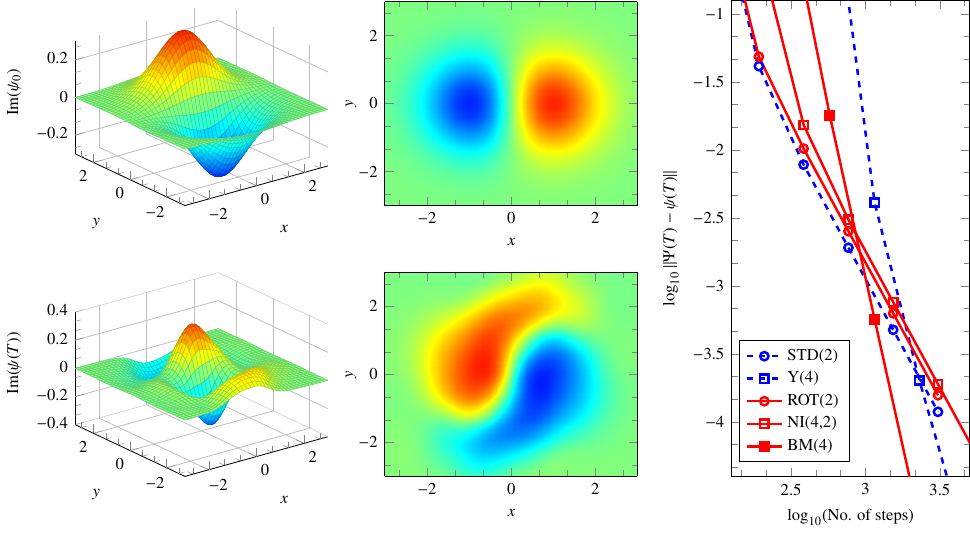}}
\caption{\label{fig3} (color online) For detailed captions, cf. Fig.~\ref{fig2}. Results for the nonlinear Hamiltonian \eqref{eq:rot_pert_h} with strong nonlinearity $g=50$ solved by \eqref{eq.nonlin2nd}. The top row shows the imaginary (left) and real part (center) of the initial condition, whereas the corresponding pictures for the exact solution at $T=5$ are displayed in the bottom row. The right panel demonstrates the smaller error constant for the proposed decomposition (red solid) in comparison with the standard split (blue dashed). The circles correspond to second order methods ROT(2) from \eqref{strang} and STD(2) \eqref{strang:classic}.}
\end{figure}

\subsection{Rotating BEC in the presence of dissipation}
Motivated by \cite{bao06dor}, we show how to adapt our method for a dissipative setup modeled by
\beq\label{eq.diss}
 (i-\lambda)\partial_t \psi = H_A(t)\psi + (g|\psi|^2  + V(t)) \psi.
\eeq
The dissipation, or loss in norm, is driven by the parameter $\lambda>0$.
Our methodology is straightforward to adapt to this setting. With the observation that the formal solution operator can be obtained by simply replacing $h$ by $ih/(i-\lambda)$ in the Magnus averaged Hamiltonian.
We conclude that the coefficients can then be computed just as before, with the difference that the polynomial system and its solutions have now become complex valued.

Due to the dissipation, Lemma~\ref{lemma} is no longer valid and the nonlinearity has to be solved differently. 
First, we stress that time is propagated together with $H_A$, containing the Laplacian and thereby recovering the Runge-Kutta-Nystr\"om (RKN) structure of the algebra\footnote{A Lie algebra generated by $A,B$ is said to be of type RKN if $[B,[B,[B,A]]]=0$.}. This has the additional benefit that more efficient splitting methods can be considered \cite{bader13itp}. Hence, as mentioned, the potential is frozen in the remaining part
\[
	(i-\lambda)\partial_t \psi = (g|\psi|^2  + V(x,t_{frozen})) \psi, \qquad \psi(x,0) = \psi_0.
\]
Using a standard trick from the Ginzburg-Landau equation, it can be solved by noting that
\[
\frac{d}{dt} |\psi|^2 = \dot{\psi}^*\psi+\psi^*\dot\psi=-\frac{2\lambda}{1+\lambda^2} \left(g|\psi|^2 + V(x,t_{frozen})\right) |\psi|^2,
\]
with solution
\[
|\psi(x,t)|^2 = \frac{|\psi_0|^2 V(x,t_{frozen})}{-g|\psi_0|^2 + \exp\left({\frac{2\lambda t V}{1+\lambda^2}}\right)\left(g|\psi_0|^2+V\right)}\ .
\]
The full remainder is then propagated as
\beq\label{eq.ginzburg}
	\psi(x,t) = \exp\left(\frac{1}{i-\lambda} \left(tV + g\int_0^t |\psi(x,s)|^2\, ds\right)
	\right)\psi_0
	= %\exp\left(\frac{1}{i-\lambda} \cdot \frac{1}{2g\lambda}\log\left(2\lambda t \phi(2\lambda t V) g|\psi_0|^2 + e^{2\lambda t V}
	\exp\left(-\frac{i+\lambda}{2\lambda} 
	\log\left[\frac{2\lambda t}{1+\lambda^2}\phi\left(\frac{2\lambda}{1+\lambda^2} t V\right) g|\psi_0|^2 + \exp\left({\frac{2\lambda}{1+\lambda^2} t V}\right)
	\right]
	\right)\psi_0,
\eeq
where $\phi(z) = (\exp(z)-1)/z$. This formulation avoids numerical singularities around $V\approx 0$ and $\psi_0\approx0$ and \eqref{eq.ginzburg} can be conveniently computed.
This equation is related to the imaginary time propagation technique to compute eigenstates of the Schr\"odinger equation \cite{chin05foa,bader13itp,bader14spi}.
Higher order splittings (including the triple jump) necessarily require negative time-steps and the consequent stability problems prohibit their use in this application.
However, complex time-steps could be -- in principle -- used to overcome this limitation.\
It has been shown in Ref.~\cite{bader14gis} how to use complex coefficients for the GPE including this variant. The findings suggest that the necessary doubling of the computational cost by introducing new variables in addition to the evaluation of at least three exponentials $\varphi^{H_A}_{a_j h}$ are less efficient than the simple use of Richardson extrapolation methods on the highly efficient new second order method with real splitting coefficients.
In Fig.~\ref{fig4_diss}, we show the performance of our method in comparison with the reference second order method STD(2) where the nonlinearity has also been propagated according to \eqref{eq.ginzburg}. The used an equidistant grid of $N_x=N_y=128$ points on $[-10,10]^2$ and integrated until $T=3$. The parameters are as in section~\ref{sec.num.weaknonlin} with $\lambda=0.02$.

\begin{figure}[!ht]%
\centering
\parbox{\textwidth}{{\includegraphics[width=\textwidth]{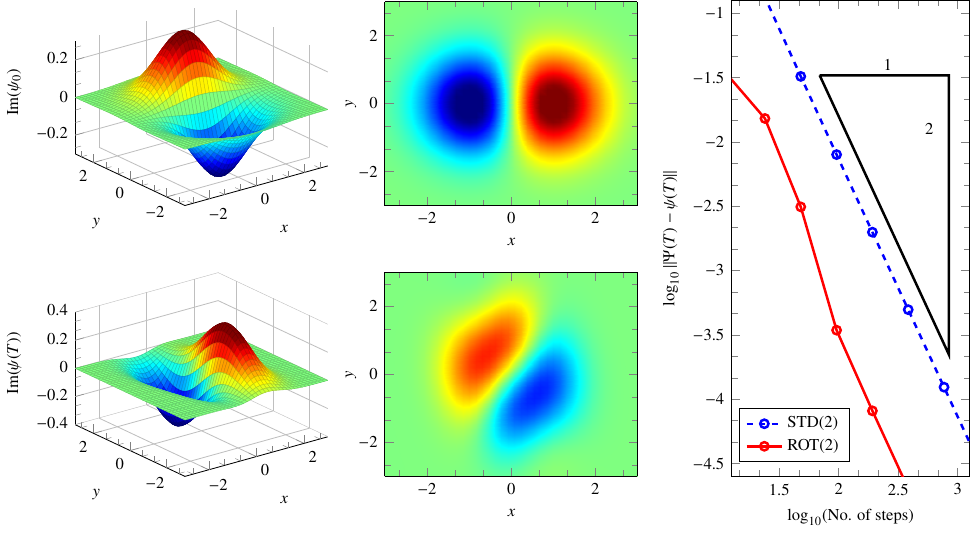}}}
\caption{\label{fig4_diss} (color online) Results for the nonlinear Hamiltonian in the presence of dissipation \eqref{eq.diss}. The top row shows the imaginary (left) and real part (center) of the initial condition, whereas the corresponding pictures for the exact solution at $T=5$ are displayed in the bottom row. The right panel demonstrates the smaller error constant for the proposed decomposition (red solid) in comparison with the standard split (blue dashed). The triangle indicates that the methods are of second order as expected.
}
\end{figure}

\section{Conclusions}
We have designed an efficient algorithm to integrate time\--de\-pen\-dent rotating BEC using only Fourier transforms and an iterative solver for a small algebraic system. 
The method solves any quadratic Hamiltonian, such as for example rotating condensates subject to time-dependent harmonic trappings, up to the desired precision and is thus particularly useful if the full Hamiltonian can be regarded as a perturbation thereof.
The method outperforms literature splitting methods for small nonlinearities at any precision and for large nonlinearieties at higher order and precision. Its adaption to dissipative Gross-Pitaevskii equation is shown to be straightforward.
The algorithm is related to splitting methods and Magnus expansions and as such preserves unitarity and gauge invariance of the exact solution operator and a corresponding modified Hamiltonian can be easily identified.
The proof technique is based on the finiteness of the Lie algebra generated by the Hamiltonian on which Magnus averaging has been performed. 
The quantum mechanical algebra has been identified with its classical mechanical counterpart from which a matrix representation has been derived.
The generalization of this technique to other quantum mechanical systems is subject of future work.

\bibliographystyle{plainnat} %acm?, plainnat for DOI

\bibliography{bibliography}

\providecommand*\hyphen{-}
\begin{thebibliography}{33}
\providecommand{\natexlab}[1]{#1}
\providecommand{\url}[1]{\texttt{#1}}
\expandafter\ifx\csname urlstyle\endcsname\relax
  \providecommand{\doi}[1]{doi: #1}\else
  \providecommand{\doi}{doi: \begingroup \urlstyle{rm}\Url}\fi

\bibitem[Anderson et~al.(1995)Anderson, Ensher, Matthews, Wieman, and
  Cornell]{bec-experiment_1}
M.~H. Anderson, J.~R. Ensher, M.~R. Matthews, C.~E. Wieman, and E.~A. Cornell.
\newblock Observation of {B}ose-{E}instein condensation in a dilute atomic
  vapor.
\newblock \emph{Science}, 269:\penalty0 198--201, 1995.
\newblock \doi{10.1126/science.269.5221.198}.

\bibitem[Bader(2014)]{bader14gis}
P.~Bader.
\newblock \emph{Geometric Integrators for Schr\"odinger equations}.
\newblock PhD thesis, Universidad Polit\'ecnica de Valencia, Valencia, Spain,
  2014.

\bibitem[Bader and Blanes(2011)]{bader11fmf}
P.~Bader and S.~Blanes.
\newblock Fourier methods for the perturbed harmonic oscillator in linear and
  nonlinear {{S}chr\"odinger} equations.
\newblock \emph{Phys. Rev. E}, 83:\penalty0 046711, 4 2011.
\newblock \doi{10.1103/PhysRevE.83.046711}.

\bibitem[Bader et~al.(2013)Bader, Blanes, and Casas]{bader13itp}
P.~Bader, S.~Blanes, and F.~Casas.
\newblock Solving the {{S}chr\"odinger} eigenvalue problem by the imaginary
  time propagation technique using splitting methods with complex coefficients.
\newblock \emph{J. Chem. Phys.}, 139:\penalty0 124117, 2013.
\newblock \doi{10.1063/1.4821126}.

\bibitem[Bader et~al.(2014)Bader, Blanes, and Ponsoda]{bader14spi}
P.~Bader, S.~Blanes, and E.~Ponsoda.
\newblock Structure preserving integrators for solving (non-)linear quadratic
  optimal control problems with applications to describe the flight of a
  quadrotor.
\newblock \emph{J. Comput. Appl. Math.}, 262\penalty0 (0):\penalty0 223 -- 233,
  2014.
\newblock \doi{10.1016/j.cam.2013.09.061}.

\bibitem[Bao and Shen(2005)]{bao05afo}
W.~Bao and J.~Shen.
\newblock A fourth-order time-splitting {L}aguerre-{H}ermite pseudo-spectral
  method for {B}ose-{E}instein condensates.
\newblock \emph{SIAM J. Sci. Comput.}, 26\penalty0 (6):\penalty0 2010--2028,
  2005.
\newblock \doi{10.1137/030601211}.

\bibitem[Bao et~al.(2006)Bao, Du, and Zhang]{bao06dor}
W.~Bao, Q.~Du, and Y.~Zhang.
\newblock Dynamics of rotating {Bose-Einstein} condensates and its efficient
  and accurate numerical computation.
\newblock \emph{SIAM J. Appl. Math.}, 66\penalty0 (3):\penalty0 758--786, 2006.
\newblock \doi{10.1137/050629392}.

\bibitem[Bao et~al.(2009)Bao, Li, and Shen]{bao09agl}
W.~Bao, H.~Li, and J.~Shen.
\newblock A generalized-{Laguerre-Fourier-Hermite} pseudospectral method for
  computing the dynamics of rotating {Bose-Einstein} condensates.
\newblock \emph{SIAM J. Sci. Comput.}, 31\penalty0 (5):\penalty0 3685--3711,
  2009.
\newblock \doi{10.1137/080739811}.

\bibitem[Bao et~al.(2013)Bao, Marahrens, Tang, and Zhang]{bao13asa}
W.~Bao, D.~Marahrens, Q.~Tang, and Y.~Zhang.
\newblock A simple and efficient numerical method for computing the dynamics of
  rotating bose--einstein condensates via rotating lagrangian coordinates.
\newblock \emph{SIAM J. Sci. Comput.}, 35\penalty0 (6):\penalty0 A2671--A2695,
  2013.
\newblock \doi{10.1137/130911111}.

\bibitem[Blanes and Moan(2002)]{blanes02psp}
S.~Blanes and P.~C. Moan.
\newblock Practical symplectic partitioned {R}unge-{K}utta and
  {R}unge-{K}utta-{N}ystr{\"o}m methods.
\newblock \emph{J. Comp. Appl. Math.}, 142:\penalty0 313--330, 2002.
\newblock \doi{10.1016/S0377-0427(01)00492-7}.

\bibitem[Blanes and Moan(2006)]{blanes06fas}
S.~Blanes and P.C. Moan.
\newblock Fourth- and sixth-order commutator-free {M}agnus integrators for
  linear and non-linear dynamical systems.
\newblock \emph{Appl. Numer. Math.}, 56:\penalty0 1519--1537, 2006.
\newblock \doi{10.1016/j.apnum.2005.11.004}.

\bibitem[Blanes et~al.(2009)Blanes, Casas, Oteo, and Ros]{blanes09tme}
S.~Blanes, F.~Casas, J.~A. Oteo, and J.~Ros.
\newblock The {M}agnus expansion and some of its applications.
\newblock \emph{Phys. Rep.}, 470:\penalty0 151--238, 2009.
\newblock \doi{10.1016/j.physrep.2008.11.001}.

\bibitem[Blanes et~al.(2010)Blanes, Diele, Marangi, and Ragni]{blanes10sac}
S.~Blanes, F.~Diele, C.~Marangi, and S.~Ragni.
\newblock Splitting and composition methods for explicit time dependence in
  separable dynamical systems.
\newblock \emph{J. Comput. Appl. Math.}, 235\penalty0 (3):\penalty0 646--659,
  2010.
\newblock \doi{10.1016/j.cam.2010.06.018}.

\bibitem[Blanes et~al.(2013)Blanes, Casas, {Farr\'es}, Laskar, Makazaga, and
  Murua]{blanes13nfo}
S.~Blanes, F.~Casas, A.~{Farr\'es}, J.~Laskar, J.~Makazaga, and A.~Murua.
\newblock New families of symplectic splitting methods for numerical
  integration in dynamical astronomy.
\newblock \emph{Appl. Num. Math.}, 68\penalty0 (0):\penalty0 58--72, 2013.
\newblock \doi{10.1016/j.apnum.2013.01.003}.

\bibitem[Bradley et~al.(1995)Bradley, Sackett, Tollett, and
  Hulet]{bec-experiment_2}
C.~C. Bradley, C.~A. Sackett, J.~J. Tollett, and R.~G. Hulet.
\newblock Evidence of {B}ose-{E}instein condensation in an atomic gas with
  attractive interactions.
\newblock \emph{Phys. Rev. Lett.}, 75:\penalty0 1687--1690, 8 1995.
\newblock \doi{10.1103/PhysRevLett.75.1687}.

\bibitem[Chin and Krotscheck(2005)]{chin05foa}
S.~A. Chin and E.~Krotscheck.
\newblock Fourth-order algorithms for solving the imaginary-time
  {G}ross-{P}itaevskii equation in a rotating anisotropic trap.
\newblock \emph{Phys. Rev. E}, 72:\penalty0 036705, 2005.
\newblock \doi{10.1103/PhysRevE.72.036705}.

\bibitem[Davis et~al.(1995)Davis, Mewes, Andrews, van Druten, Durfee, Kurn, and
  Ketterle]{bec-experiment_3}
K.~B. Davis, M.~O. Mewes, M.~R. Andrews, N.~J. van Druten, D.~S. Durfee, D.~M.
  Kurn, and W.~Ketterle.
\newblock {B}ose-{E}instein condensation in a gas of sodium atoms.
\newblock \emph{Phys. Rev. Lett.}, 75:\penalty0 3969--3973, 11 1995.
\newblock \doi{10.1103/PhysRevLett.75.3969}.

\bibitem[Faou(2012)]{faou12gni}
E.~Faou.
\newblock \emph{Geometric Numerical Integration and Schr\"odinger Equations}.
\newblock Zurich Lectures in Advanced Mathematics. Europ. Math. Soc., Z\"urich,
  2012.
\newblock \doi{10.4171/100}.

\bibitem[Fernandez et~al.(1989)Fernandez, Micha, and Echave]{fernandez89mtp}
F.~M. Fernandez, D.~A. Micha, and J.~Echave.
\newblock Molecular transition probabilities for time-dependent, bilinear
  hamiltonians in many dimensions: A recursive procedure.
\newblock \emph{Phys. Rev. A}, 40:\penalty0 74--79, 1989.
\newblock \doi{10.1103/PhysRevA.40.74}.

\bibitem[Hairer et~al.(2006)Hairer, Lubich, and Wanner]{hairer06gni}
E.~Hairer, C.~Lubich, and G.~Wanner.
\newblock \emph{Geometric Numerical Integration: Structure-Preserving
  Algorithms for Ordinary Differential Equations}.
\newblock Springer Verlag, Berlin, 2nd edition, 2006.

\bibitem[Hofst{\"a}tter et~al.(2014)Hofst{\"a}tter, Koch, and
  Thalhammer]{hofstaetter14cos}
H.~Hofst{\"a}tter, O.~Koch, and M.~Thalhammer.
\newblock Convergence analysis of high-order time-splitting pseudo-spectral
  methods for rotational {G}ross\--{P}itaevskii equations.
\newblock \emph{Numer. Math}, 127:\penalty0 315--364, 2014.
\newblock \doi{10.1007/s00211-013-0586-9}.

\bibitem[Iserles and {N\o rsett}(1999)]{iserles99ots}
A.~Iserles and S.~P. {N\o rsett}.
\newblock On the solution of linear differential equations in {L}ie {G}roups.
\newblock \emph{Phil. Trans. R. Soc. A}, 357:\penalty0 983--1019, 1999.

\bibitem[Magnus(1954)]{magnus54ote}
W.~Magnus.
\newblock On the exponential solution of differential equations for a linear
  operator.
\newblock \emph{Commun. Pure Appl. Math.}, 7:\penalty0 649--673, 1954.

\bibitem[McLachlan(1995)]{mclachlan95cmi}
R.~I. McLachlan.
\newblock Composition methods in the presence of small parameters.
\newblock \emph{BIT Numer. Math.}, 35\penalty0 (2):\penalty0 258--268, 1995.
\newblock \doi{10.1007/BF01737165}.

\bibitem[Ming et~al.(2014)Ming, Tang, and Zhang]{ming14aes}
J.~Ming, Q.~Tang, and Y.~Zhang.
\newblock An efficient spectral method for computing dynamics of rotating
  two-component {Bose-Einstein} condensates via coordinate transformation.
\newblock \emph{J. Comput. Phys.}, 258\penalty0 (0):\penalty0 538--554, 2014.
\newblock ISSN 0021-9991.
\newblock \doi{10.1016/j.jcp.2013.10.044}.

\bibitem[P\'erez-Garc\'ia and Liu(2003)]{perez03amc}
V.~M. P\'erez-Garc\'ia and X.~Liu.
\newblock Numerical methods for the simulation of trapped nonlinear
  {Schr\"odinger} systems.
\newblock \emph{Appl. Math. Comp.}, 144:\penalty0 215--235, 2003.
\newblock \doi{10.1016/S0096-3003(02)00402-2}.

\bibitem[Recamier et~al.(1985)Recamier, Micha, and Gazdy]{recamier85eti}
J.~Recamier, D.~A. Micha, and B.~Gazdy.
\newblock Energy transfer in collisions between two vibrating molecules.
\newblock \emph{Chem. Phys. Lett.}, 119\penalty0 (5):\penalty0 383--387, 1985.
\newblock \doi{10.1016/0009-2614(85)80439-5}.

\bibitem[Suzuki(1990)]{suzuki90fdo}
M.~Suzuki.
\newblock Fractal decomposition of exponential operators with applications to
  many-body theories and {M}onte {C}arlo simulations.
\newblock \emph{Phys. Lett. A}, 146:\penalty0 319--323, 1990.
\newblock \doi{10.1016/0375-9601(90)90962-N}.

\bibitem[Thalhammer et~al.(2009)Thalhammer, Caliari, and
  Neuhauser]{thalhammer09jcp}
M.~Thalhammer, M.~Caliari, and C.~Neuhauser.
\newblock High-order time-splitting {H}ermite and {F}ourier spectral methods.
\newblock \emph{J. Comput. Phys.}, 228\penalty0 (3):\penalty0 822--832, 2009.
\newblock \doi{10.1016/j.jcp.2008.10.008}.

\bibitem[Wang(2007)]{wang07ats}
H.~Wang.
\newblock A time-splitting spectral method for coupled {Gross-Pitaevskii}
  equations with applications to rotating {Bose-Einstein} condensates.
\newblock \emph{J. Comput. Appl. Math.}, 205\penalty0 (1):\penalty0 88 -- 104,
  2007.
\newblock ISSN 0377-0427.
\newblock \doi{j.cam.2006.04.042}.

\bibitem[Wei and Norman(1963)]{wei63las}
J.~Wei and E.~Norman.
\newblock Lie algebraic solution of linear differential equations.
\newblock \emph{J. Math. Phys.}, 4:\penalty0 575--581, 1963.
\newblock \doi{10.1063/1.1703993}.

\bibitem[Wilcox(1967)]{wilcox67jmp}
R.M. Wilcox.
\newblock Exponential operators and parameter differentiation in quantum
  physics.
\newblock \emph{J. Math. Phys.}, 8:\penalty0 962--982, 1967.
\newblock \doi{10.1063/1.1705306}.

\bibitem[Yoshida(1990)]{yoshida90coh}
H.~Yoshida.
\newblock Construction of higher order symplectic integrators.
\newblock \emph{Phys. Lett. A}, 150:\penalty0 262--268, 1990.
\newblock \doi{10.1016/0375-9601(90)90092-3}.

\end{thebibliography}
\end{document}